\documentclass[12pt]{amsart}

\usepackage[english]{babel}
\usepackage[utf8x]{inputenc}
\usepackage[T1]{fontenc}
\usepackage{amsmath}
\usepackage{amstext}
\usepackage{amsfonts}
\usepackage{amssymb}
\usepackage{amsthm}
\usepackage{amsrefs}
\usepackage{mathtools}
\usepackage{dsfont}
\usepackage{color}
\usepackage{graphicx}
\usepackage[colorinlistoftodos]{todonotes}
\usepackage[colorlinks=true, allcolors=blue]{hyperref}
\usepackage{float}
\allowdisplaybreaks
\newtheorem{theorem}{Theorem}[section]
\newtheorem{lemma}[theorem]{Lemma}

\newtheorem{cor}[theorem]{Corollary}
\newtheorem{thm}[theorem]{Theorem}
\newtheorem{lem}[theorem]{Lemma}

\newtheorem*{cor*}{Corollary}
\newtheorem*{thm*}{Theorem}
\newtheorem*{lem*}{Lemma}

\newtheorem*{prop*}{Proposition}
\theoremstyle{definition}

\newtheorem*{defn*}{Definition}

\theoremstyle{remark}
\newtheorem{remark}[theorem]{Remark}

\newcommand{\pr}{\operatorname{Prob}}

\newcommand{\act}{\curvearrowright}

\newcommand{\cA}{\mathcal{A}}
\newcommand{\cB}{\mathcal{B}}

\newcommand{\bB}{{\mathbb{B}}}
\newcommand{\bC}{{\mathbb{C}}}
\newcommand{\bE}{{\mathbb{E}}}
\newcommand{\bF}{{\mathbb{F}}}

\newcommand{\bN}{{\mathbb{N}}}

\newcommand{\bR}{{\mathbb{R}}}

\newcommand{\om}{\omega}

\newcommand{\supp}{\operatorname{Supp}}

\newcommand{\ep}{\varepsilon}

\title{On simplicity of intermediate $C^*$-algebras}

\author[T. Amrutam]{Tattwamasi Amrutam}
\author[M. Kalantar]{Mehrdad Kalantar}
\address{Department of Mathematics\\ University of Houston\\USA}
\email{tamrutam@math.uh.edu}
\email{kalantar@math.uh.edu}

\date{}

\begin{document}

\begin{abstract}
We prove simplicity of all intermediate $C^*$-algebras $C^*_{r}(\Gamma)\subseteq \cB \subseteq \Gamma\ltimes_r C(X)$ in the case of minimal actions of $C^*$-simple groups $\Gamma$ on compact spaces $X$. For this, we use the notion of stationary states, recently introduced by Hartman and Kalantar in \cite{HartKal}. We show that the Powers' averaging property holds for the reduced crossed product $\Gamma\ltimes_r \cA$ for any action $\Gamma\act \cA$ of a $C^*$-simple group $\Gamma$ on a unital $C^*$-algebra $\cA$, and use it to prove a one-to-one correspondence between stationary states on $\cA$ and those on $\Gamma\ltimes_r \cA$.
\end{abstract}

\thanks{MK was partially supported by the NSF Grant DMS-1700259.}

\maketitle

\section{Introduction and statement of the main results}

Crossed product $C^*$-algebras are among the most important and most studied classes of $C^*$-algebras. They provide deep connections between theories of  $C^*$-algebras and dynamical systems. The problem of simplicity of reduced crossed product $C^*$-algebras, and more generally understanding the ideal structure of the full crossed products
have received lots of attention in the past few decades (see e.g. \cites{Tak64, Elliott80, DS86, KT90, Bedos91, JangLee93, AS94, Sier10}). 

In \cite{DS86}, de la Harpe and Skandalis proved  
that for any action $\Gamma\act \cA$ of a Powers' group $\Gamma$ on a unital $C^*$-algebra $\cA$, the reduced crossed product $\Gamma\ltimes_r \cA$ is simple if $\cA$ is $\Gamma$-simple (i.e. has no non-zero proper two sided closed $\Gamma$-invariant ideals).  
They left it as a question whether the same holds in the more general case of $C^*$-simple groups. In \cite{BKKO}, the authors answered this question by proving the result for all $C^*$-simple groups, by using the dynamics of the Furstenberg boundary action. 

Intermediate $C^*$-algebras, i.e. $C^*$-algebras $\cB$ of the form $C_{r}^*(\Gamma) \subseteq \cB \subseteq \Gamma\ltimes_r \cA$, have recently gained some particular attention, for instance in the work of Suzuki (\cites{Suz17, Suz18}) in connection to problems of minimal ambient nuclear $C^*$-algebras as well as maximal injective von Neumann subalgebras.

In this paper we consider the simplicity problem for  intermediate $C^*$-subalgebras of crossed products of $C^*$-simple group actions, and more generally, the $\Gamma$-simplicity of their unital $\Gamma$-invariant $C^*$-subalgebras. 

\begin{theorem}\label{main-general}
Let $\Gamma$ be a countable discrete $C^*$-simple group, and let $\cA$ be a $\Gamma$-$C^*$-algebra. Suppose for some $C^*$-simple measure $\mu\in\pr(\Gamma)$, all $\mu$-stationary states on $\cA$ are faithful. Then any unital $\Gamma$-invariant $C^*$-subalgebra of the reduced crossed product $\Gamma\ltimes_r \cA$ is $\Gamma$-simple.
\end{theorem}

Since, with respect to the inner action of $\Gamma$, any $C^*$-subalgebra $\cB$ of $\Gamma\ltimes_r \cA$ that contains $C^*_{r}(\Gamma)$, as well as any ideal in $\cA$, are invariant, the following is immediate.

\begin{cor}
Under the assumptions of Theorem \ref{main-general}, any intermediate $C^*$-subalgebra $C_{r}^*(\Gamma) \subseteq \cB \subseteq \Gamma\ltimes_r \cA$ is simple.
\end{cor}

In the case of commutative $\Gamma$-$C^*$-algebras $\cA= C(X)$, we obtain the following.

\begin{thm}\label{main-commutative}
Let $\Gamma$ be a countable discrete $C^*$-simple group, and let $\Gamma \curvearrowright X$ be a minimal action of $\Gamma$ on a compact space $X$. Then any unital $\Gamma$-invariant $C^*$-subalgebra of $\Gamma\ltimes_r \cA$ is $\Gamma$-simple. In particular, any intermediate $C^*$-subalgebra $C_{r}^*(\Gamma) \subseteq \cB \subseteq \Gamma\ltimes_r C(X)$ is simple.
\end{thm}

Examples of actions $\Gamma\act \cA$, where $\cA$ is noncommutative and assumptions of Theorem \ref{main-general} hold, include $\Gamma\act C^*_{r}(\Gamma)$ by inner automorphisms, for any $C^*$-simple group $\Gamma$ (\cite[Theorem 5.1]{HartKal}), as well as $\bF_n\act\bF_n\ltimes_r C(\partial\bF_n)$, also by inner automorphisms, for any $n\geq 2$ (\cite[Example 4.13]{HartKal}).

None of the proofs in \cite{DS86} and \cite{BKKO} of simplicity of the reduced crossed products have obvious modification to include the case of invariant subalgebras. In fact, one can observe that such a result is very far from being true in general. For example, let $\cA$ be a non-trivial simple $C^*$-algebra and let $\Gamma\act \cA$ be the trivial action of a Powers' group $\Gamma$. Then $\Gamma\ltimes_r\cA = C^*_{r}(\Gamma)\otimes \cA$ is simple. However, if $\cB$ is a non-simple unital $C^*$-subalgebra of $\cA$, then $C^*_{r}(\Gamma)\subset\Gamma\ltimes_r\cB \subset \Gamma\ltimes_r\cA$, and $\Gamma\ltimes_r\cB = C^*_{r}(\Gamma)\otimes \cB$ is not simple.
It is not hard to construct even a faithful such action. But one should notice that the main reason that simplicity for invariant subalgebras could fail is that, in general, $\Gamma$-simplicity does not pass to subalgebras. 

On the other hand, in the above setup, even in the more general case of a $C^*$-simple group $\Gamma$, if $\cA = C(X)$ is commutative and $\Gamma$-simple (which is equivalent to minimality of $\Gamma\act X$), since any invariant $C^*$-subalgebra $\cB\subset \cA$ is of the form $C(Y)$ where $Y$ is an equivariant factor of $X$, and since minimality passes to factors, it follows $\Gamma\ltimes_r\cB$ is also simple by \cite[Theorem 7.1]{BKKO}.

Thus, we have observed that if $\Gamma$ is $C^*$-simple and $\Gamma\act X$ is minimal, then any intermediate $C^*$-subalgebra $C^*_{r}(\Gamma)\subset\cB\subset \Gamma\ltimes_r C(X)$ which itself is of the form $\cB= \Gamma\ltimes_r C(Y)$, is simple.

To deal with general intermediate $C^*$-subalgebras, not necessarily of the crossed product type, we need to translate minimality in the non-commutative setting in a way that passes to subalgebras and does not require a crossed product structure to realize.
Inspired by the recent work \cite{HartKal} of Hartman and the second named author, we use the notion of stationary states to capture ``minimality'' of the intermediate $C^*$-subalgebras.\\

Before proceeding into the details of our results, we recall some definitions and basic facts, and fix some conventions and terminology. Unless otherwise stated, $\Gamma$ will be a countable discrete group, and all compact spaces are assumed Hausdorff. 
We denote by $\lambda: \Gamma \to \cB(\ell^2(\Gamma))$ the left regular representation of $\Gamma$, and by $C^*_{r}(\Gamma)$ the reduced $C^*$-algebra of $\Gamma$, i.e. the $C^*$-algebra generated by $\{\lambda_s: s\in \Gamma\}$. We denote by $\tau_0$ the canonical trace on $C^*_{r}(\Gamma)$, defined by $\tau_0(\lambda_e) = 1$ and $\tau_0(\lambda_s) = 0$ for all non-trivial elements $s\in \Gamma/\{e\}$.

By \emph{$\Gamma$-$C^*$-algebra}, we mean a unital $C^*$-algebra on which $\Gamma$ acts by $*$-automorphisms. We say $\cA$ is $\Gamma$-simple if it does not contain any non-trivial proper closed $\Gamma$-invariant ideals. If $\cA = C(X)$ is commutative, then $\cA$ is $\Gamma$-simple iff $\Gamma\act X$ is minimal, i.e. the compact $\Gamma$-space $X$ does not have any non-empty proper closed $\Gamma$-invariant subsets.

Any action $\Gamma\act \cA$ induces an action of $\Gamma$ on the state space of $\cA$ in a canonical way, $(s\tau)(a) = \tau(s^{-1}(a))$ for $s\in\Gamma$, $a\in \cA$, and a state $\tau$ on $\cA$. Let $\mu\in\pr(\Gamma)$, a state $\tau$ on a $\Gamma$-$C^*$-algebra $\cA$ is said to be \emph{$\mu$-stationary} if $\mu*\tau = \tau$, where $\mu*\tau = \sum_{s\in\Gamma} \mu(s)s\tau$ is the convolution of $\tau$ by $\mu$. The theory of stationary states was introduced and studied in \cite{HartKal}, where applications to several rigidity problems in ergodic theory and operator algebras were given. One of the main results there states that a countable group $\Gamma$ is $C^*$-simple if and only if there is a measure $\mu\in\pr(\Gamma)$ such that the canonical trace $\tau_0$ is the unique $\mu$-stationary state on $C^*_{r}(\Gamma)$ (\cite[Theorem 5.1]{HartKal}); such a measure $\mu$ is called a \emph{$C^*$-simple measure} (\cite[Definition 5.2]{HartKal}).

\section{Powers' averaging property for crossed products}
In this section we prove the Powers' averaging property for the reduced crossed product $\Gamma\ltimes_r \cA$, in the case of an action $\Gamma\act \cA$ of a $C^*$-simple group $\Gamma$ on a unital $C^*$-algebra $\cA$. In the case of Powers' groups $\Gamma$, this result was proved by de la Harpe and Skandalis in \cite{DS86}, and it was used to prove simplicity of the reduced crossed product $\Gamma\ltimes_r \cA$ when $\cA$ is $\Gamma$-simple.
Recent developments of the subject include two independent work of Haagerup \cite{Haagerup} and Kennedy \cite{Kennedy}, where they prove that the reduced $C^*$-algebra $C^*_{r}(\Gamma)$ of any $C^*$-simple group $\Gamma$ has the Powers' averaging property. Below we prove the same averaging scheme can be lifted to the crossed product level as well.

First, let us quickly recall the construction of reduced crossed products in order to introduce our notation. Let $\cA$ be a unital $\Gamma$-$C^*$-algebra. Fix a faithful $*$-representation $\pi: \cA \to \mathbb{B}(H)$ of $\cA$ into the space of bounded operators on the Hilbert space $H$. Denote by $\ell^2(\Gamma, H)$ the space of square summable $H$-valued functions on $\Gamma$. The group $\Gamma$ acts on $\ell^2(\Gamma, H)$ by left translation unitaries 
\[
\tilde\lambda_s\xi(t):=\xi(s^{-1}t)\quad \left(s,t \in \Gamma, \,\xi \in\ell^2(\Gamma,H)\right).
\]
There is also a $*$-representation $\sigma:\cA \to \cB(\ell^2(\Gamma, H))$ defined by
\[
[\sigma(a)\xi](t):=\pi(t^{-1}a)(\xi(t))\quad \left(a \in A,\, \xi \in\ell^2(\Gamma,H),\, t \in \Gamma \right) .
\]
The reduced crossed product $\Gamma\ltimes_r\cA$ is the $C^*$-algebra generated by unitaries $\{\tilde\lambda_s : s\in\Gamma\}$ and operators $\{\sigma(a) : a\in\cA\}$ in $\cB(\ell^2(\Gamma, H))$. Note that $\tilde\lambda_s\sigma(a)\tilde\lambda_{s^{-1}}=\sigma(sa)$ for all $s \in \Gamma$ and $a \in A$. In particular, the group $\Gamma$ also acts on $\Gamma\ltimes_r\cA$ by inner automorphisms.

We denote by $\bE: \Gamma\ltimes_r\cA \to \sigma(\cA)$ the canonical conditional expectation which is defined by $\bE(\sigma(a)) = \sigma(a)$ and $\bE(\sigma(a)\tilde\lambda_s) = 0$ for $a\in \cA$ and $s \in \Gamma/\{e\}$. The map $\bE$ is $\Gamma$-equivariant and faithful.

The following lemma provides the estimation that will allow us to lift an averaging scheme from the reduced $C^*$-algebra to the reduced crossed product.

\begin{lemma}\label{ineq-av-1}
Let $\Gamma$ be a discrete group, and let $\cA$ be a $\Gamma$-$C^*$-algebra. Then for any $t_0, s_1, \dots, s_m\in \Gamma$, $p_1, \dots, p_m \in \bR^+$, and $a\in \cA$ we have
\begin{equation}\label{ave-ineq-1}
\left\|\sum_{j=1}^m p_j\tilde\lambda_{s_j}\sigma(a)\tilde\lambda_{t_0}\tilde\lambda_{s_j^{-1}}\right\|_{\bB(\ell^2(\Gamma, H))} \le \|a\|_\cA \left\|\sum_{j=1}^m p_j\lambda_{s_jt_0s_j^{-1}}\right\|_{\bB(\ell^2(\Gamma))} .
\end{equation}
\end{lemma}

\begin{proof}
For $\xi \in \ell^2(\Gamma, H)$, observe that for each $t\in \Gamma$ we have
\[\begin{split}
&\sum_{j=1}^m p_j[\tilde\lambda_{s_j}\sigma(a)\tilde\lambda_{t_0}\tilde\lambda_{s_j^{-1}}(\xi)](t)
=\sum_{j=1}^m p_j[\sigma(s_ja)\tilde\lambda_{s_jt_0s_j^{-1}}\xi](t)
\\&=\sum_{j=1}^m p_j[\sigma(s_ja)\xi](s_jt_0^{-1}s_j^{-1}t)
=\sum_{j=1}^m p_j\pi(t^{-1}s_jt_0a)[\xi(s_jt_0^{-1}s_j^{-1}t)] .
\end{split}\]
Define the function $\xi_1(t)=\|\xi(t)\|_H$, $t\in \Gamma$. Then $\xi_1\in\ell^2(\Gamma)$, and $\left\|\xi_1\right\|_{\ell^2(\Gamma)}=\left\|\xi\right\|_{\ell^2(\Gamma, H)}$. We have
\[\begin{split}
\left\|\sum_{j=1}^m p_j\tilde\lambda_{s_j}\sigma(a)\tilde\lambda_{t_0}\tilde\lambda_{s_j^{-1}}(\xi)\right\|_{\ell^2(\Gamma, H)}^2
&=\sum_{t \in \Gamma}\left\|\sum_{j=1}^m p_j\pi(t^{-1}s_jt_0a)[\xi(s_jt_0^{-1}s_j^{-1}t)]\right\|_H^2
\\&\le \|a\|_\cA^2\sum_{t \in \Gamma}\left(\sum_{j=1}^m p_j\left\|\xi(s_jt_0^{-1}s_j^{-1}t)\right\|_H\right)^2
\\&=\|a\|_\cA^2\left\|\sum_{j=1}^m p_j\lambda_{s_jt_0s_j^{-1}}(\xi_1)\right\|_{\ell^2(\Gamma)}^2
\\&\leq\|a\|_\cA^2\left\|\sum_{j=1}^m p_j\lambda_{s_jt_0s_j^{-1}}\right\|_{\bB(\ell^2(\Gamma))}^2\left\|\xi_1\right\|_{\ell^2(\Gamma)}^2 ,
\end{split}\]
and since $\left\|\xi_1\right\|_{\ell^2(\Gamma)}=\left\|\xi\right\|_{\ell^2(\Gamma, H)}$, the inequality \eqref{ave-ineq-1} follows.
\end{proof}

It follows, in particular, from the above Lemma \ref{ineq-av-1} that if $C^*_{r}(\Gamma)$ has the Powers' averaging property, then so does the reduced crossed product $\Gamma\ltimes_r \cA$ for any action $\Gamma\act \cA$. But in order to prove our characterization of stationary states on the crossed product, we need a more precise averaging scheme.

We denote by $\mu^*k$ the $k$-th convolution power of a measure $\mu\in\pr(\Gamma)$. Also, for $a\in \cA$ and a measure $\mu'\in\pr(\Gamma)$, we denote $\mu'*a = \sum_{t\in\Gamma} \mu'(t) ta$ for the convolution of $a$ by $\mu'$.

\begin{theorem}\label{Pow-av-red-cros-prod}
Let $\Gamma$ be a $C^*$-simple group, let $\mu \in \pr(\Gamma)$ be a $C^*$-simple measure, and let $\cA$ be a $\Gamma$-$C^*$-algebra. Then 
\[
\left\|\, \frac1n\sum_{k=1}^n\mu^k*(a-\bE(a)) \, \right\| \xrightarrow{n \to \infty} 0
\]
for every $a\in \Gamma\ltimes_r \cA$.
\end{theorem}

\begin{proof}
Let $a\in \Gamma\ltimes_r \cA$, and let $\ep>0$ be given. Then, there are $t_1, \dots, t_m \in \Gamma\backslash \{e\}$ and $a_1, \dots, a_m\in \cA$ such that for $b=\sum_{i=1}^m\sigma(a_{i})\tilde\lambda_{t_i}+E(a)$ we have $\|b-a\|_{\Gamma\ltimes_r \cA} < \frac{\ep}{2}$. Since $\mu$ is $C^*$-simple, it follows from \cite[Proposition 4.7]{HartKal} that $\left\|\frac{1}{n}\sum_{k=1}^n\mu^n*\lambda_{t_i}\right\|_{C^*_{r}(\Gamma)} \to 0$, as $n\to\infty$, for all $i = 1, 2, \ldots, m$. Thus, Lemma \ref{ineq-av-1} implies 
\[\begin{split}
&\left\|\frac{1}{n}\sum_{k=1}^n\mu^k*(b-\bE(a))\right\|_{\Gamma\ltimes_r \cA}
=
\left\|\frac{1}{n}\sum_{k=1}^n\sum_{i=1}^m\mu^k*(\sigma(a_{i})\tilde\lambda_{t_i})\right\|_{\Gamma\ltimes_r \cA}
\\&\quad\quad\quad\quad\quad\leq
\sum_{i=1}^m \left(\|a_i\|_\cA \left\|\frac{1}{n}\sum_{k=1}^n\mu^k*\tilde\lambda_{t_i}\right\|_{C^*_{r}(\Gamma)} \right)
\xrightarrow{n\to\infty} 0 .
\end{split}\]
Hence
\[
\limsup_n\left\|\, \frac1n\sum_{k=1}^n\mu^k*(a-\bE(a)) \, \right\|_{\Gamma\ltimes_r \cA} \leq \ep,
\]
and since $\ep$ was arbitrary, the theorem follows.
\end{proof}

\section{Stationary states on the reduced crossed product}
In this section we prove for an action $\Gamma\act \cA$ of a $C^*$-simple group, a one-to-one correspondence between stationary states on $\cA$ and stationary states on the reduced crossed product $\Gamma\ltimes_r \cA$.
This correspondence, together with the important feature of stationary states that for any action $\Gamma\act \cA$ and $\mu\in\pr(\Gamma)$ there is a $\mu$-stationary state $\tau$ on $\cA$ (\cite[Proposition 4.2]{HartKal}), are the main ingredients in proving our main result, Theorem \ref{main-general}.

\begin{theorem}\label{thm:charac-stationary-crsed-prod}
Let $\Gamma$ be a $C^*$-simple group, let $\mu \in \pr(\Gamma)$ be a $C^*$-simple measure, and let $\cA$ be a $\Gamma$-$C^*$-algebra. Then any $\mu$-stationary state $\tau$ on $\Gamma\ltimes_r \cA$ is of the form $\tau=\nu\circ\sigma^{-1} \circ\bE$ for some $\mu$-stationary state $\nu$ on $\cA$.
\end{theorem}

\begin{proof}
Let $\mu \in \pr(\Gamma)$ be a $C^*$-simple measure, and let $\tau$ be a $\mu$-stationary state on $\Gamma\ltimes_r \cA$. Then, for any $a\in \Gamma\ltimes_r \cA$, Theorem \ref{Pow-av-red-cros-prod} implies
\[\begin{split}
\left|\, \tau(a-\bE(a)) \,\right| &= \left| (\mu^n*\tau) (a-\bE(a)) \,\right| 
=\left|\, \tau(\mu^n*(a-\bE(a))) \,\right| \\&\le \left\|\mu^n*(a-\bE(a))\right\|\xrightarrow{n \to \infty} 0 ,
\end{split}\]
which implies $\tau = \tau\circ \bE$. Thus, if we let $\nu=\tau|_{\sigma(\cA)}\circ\sigma$ be the state on $\cA$ obtained from restriction of $\tau$ to $\sigma(\cA)\subset \Gamma\ltimes_r \cA$, we see that $\nu$ is $\mu$-stationary and $\tau =\nu\circ\sigma^{-1} \circ\bE$.
\end{proof}

\begin{remark} 
A similar correspondence between invariant probabilities on $X$ and traces on the crossed product was proved by de la Harpe and Skandalis in \cite{DS86} in the case of minimal actions of Powers' groups. 
\end{remark}

\begin{remark}
The conclusion of the above Theorem \ref{thm:charac-stationary-crsed-prod} in the case of trivial action $\cA = \bC$ translates to unique stationarity of the canonical trace on the reduced $C^*$-algebra $C^*_{r}(\Gamma)$. Thus, it generalizes one direction of \cite[Theorem 5.1]{HartKal}, and in fact, combined with the latter, they give a similar characterization of $C^*$-simplicity, which we record in the following theorem. 

\begin{thm}
The following are equivalent for a countable group $\Gamma$.
\begin{enumerate}
\item
$\Gamma$ is $C^*$-simple;
\item
there is $\mu\in\pr(\Gamma)$ such that for any action $\Gamma\act \cA$, any $\mu$-stationary state $\tau$ on $\Gamma\ltimes_r \cA$ is of the form $\tau=\nu\circ\sigma^{-1} \circ\bE$ for some $\mu$-stationary state $\nu$ on $\cA$;
\item
there is an action $\Gamma\act \cA$ such that for some $\mu\in\pr(\Gamma)$, every $\mu$-stationary state $\tau$ on $\Gamma\ltimes_r \cA$ is of the form $\tau=\nu\circ\sigma^{-1} \circ\bE$ for some $\mu$-stationary state $\nu$ on $\cA$.
\end{enumerate}
\end{thm}
\begin{proof}
By \cite[Theorem 5.1]{HartKal} every $C^*$-simple group admits a $C^*$-simple measure, thus (1) $\implies$ (2) follows from Theorem \ref{thm:charac-stationary-crsed-prod}. The implication (2) $\implies$ (3) is trivial. Now suppose (3) holds. Then let $\eta$ be a $\mu$-stationary state on $C^*_{r}(\Gamma)$. By \cite[Proposition 4.2]{HartKal}, $\eta$ extends to a $\mu$-stationary state $\tau$ on $\Gamma\ltimes_r \cA$. Let $\nu$ be the state on $\cA$ such that $\tau=\nu\circ\sigma^{-1} \circ\bE$. Then for $s\in\Gamma/\{e\}$ we have $\eta(\lambda_s) = \nu \circ\sigma^{-1}(\bE(\lambda_s)) = 0$, hence $\eta= \tau_0$. This shows that $\tau_0$ is the unique $\mu$-stationary state on $C^*_{r}(\Gamma)$, and thus $\Gamma$ is $C^*$-simple by \cite[Theorem 5.1]{HartKal}.
\end{proof}
\end{remark}

\section{Proofs of the main results}
In this section we prove Theorems \ref{main-general} and \ref{main-commutative}.\\

\noindent
{\it Proof of Theorem \ref{main-general}.}\ 
Let $\Gamma$ be a countable discrete $C^*$-simple group, and let $\cA$ be a $\Gamma$-$C^*$-algebra. Let $\mu\in\pr(\Gamma)$ be such that all $\mu$-stationary states on $\cA$ are faithful. 
Let $\cB$ be a unital $\Gamma$-invariant $C^*$-subalgebra of $\Gamma\ltimes_r\cA$, and let $I$ be a proper closed two-sided $\Gamma$-invariant ideal of $\cB$. Then the action $\Gamma \curvearrowright \mathcal{B}$ induces an action $\Gamma \curvearrowright \mathcal{B}/I$. By \cite[Proposition 4.2]{HartKal}, there exists a $\mu$-stationary state $\eta$ on $\mathcal{B}/I$. Composing $\eta$ with the canonical quotient map $\mathcal{B}\to \mathcal{B}/I$ we obtain a $\mu$-stationary state $\tilde\eta$ on $\mathcal{B}$ that vanishes on $I$. Now by the same \cite[Proposition 4.2]{HartKal}, this $\tilde\eta$ can be extended to a $\mu$-stationary state $\tau$ on $\Gamma\ltimes_r\cA$. By Theorem \ref{thm:charac-stationary-crsed-prod}, there is a $\mu$-stationary state $\nu$ on $\cA$ such that $\tau=\nu\circ\sigma^{-1} \circ\bE$. By the assumptions, $\nu$ is faithful, and since $\bE$ is also faithful, it follows $\tau$ is faithful. But $\tau$ vanishes on $I$, hence $I$ is trivial.\qed\\

In order to prove Theorem \ref{main-commutative} we need to work with a generating $C^*$-simple measure, existence of which for a $C^*$-simple group was not established formally in \cite{HartKal}. But we verify below that a simple tweak in the proof of \cite[Theorem 5.1]{HartKal} will do the job.

\begin{lem}{(cf. \cite[Theorem 5.1]{HartKal})}\label{lem:gen-C*-simple-measure}
Every countable $C^*$-simple group $\Gamma$ admits a generating $C^*$-simple measure.
\end{lem}

\begin{proof}
It was shown in the proof of \cite[Theorem 5.1]{HartKal} that if $\Gamma$ is a $C^*$-simple group then there is a sequence $(\mu_n)$ of probabilities on $\Gamma$ such that $\left\|\mu_n * a -\tau_0(a)1_{C^*_{r}(\Gamma)}\right\| \to 0$, as $n\to\infty$, for all $a\in C^*_{r}(\Gamma)$, and that any such sequence $(\mu_n)$ has a subsequence $(\mu_{n_k})$ such that $\mu := \sum_{k=1}^\infty \frac{1}{2^k} \mu_{n_k}$ is a $C^*$-simple measure.

Now, consider a sequence $(\mu_n)$ as above, and for a fixed $\om\in \pr(\Gamma)$ with full support, let $\tilde\mu_n := \om*\mu_n$ for each $n\in \bN$. Then every $\tilde\mu_n$ has full support, and
\[\begin{split}
\left\|\tilde\mu_n * a -\tau_0(a)1_{C^*_{r}(\Gamma)}\right\| 
&= \left\|\om*\mu_n * a -\tau_0(a)1_{C^*_{r}(\Gamma)}\right\|
\\&= \left\|\om*[\mu_n * a -\tau_0(a)1_{C^*_{r}(\Gamma)}]\right\|
\\&\le \left\|\mu_n * a -\tau_0(a)1_{C^*_{r}(\Gamma)}\right\|
\xrightarrow{n\to\infty} 0
\end{split}\]
for all $a\in C^*_{r}(\Gamma)$, which implies, as commented above, that for an appropriately chosen subsequence, the measure $\tilde\mu := \sum_{k=1}^\infty \frac{1}{2^k} \tilde\mu_{n_k}$ is $C^*$-simple. Since the measures $\mu_{n_k}$ have full support, so does the $C^*$-simple measure $\tilde\mu$.
\end{proof}

\noindent
{\it Proof of Theorem \ref{main-commutative}.}\ 
Let $\Gamma$ be a countable discrete $C^*$-simple group, and let $\Gamma\act X$ be a minimal action on the compact space $X$. By Lemma \ref{lem:gen-C*-simple-measure}, there is a generating $C^*$-simple measure $\mu$ on $\Gamma$. Let $\nu\in\pr(X)$ be $\mu$-stationary. It is not hard to see that $\supp(\nu)$ is invariant under the action of elements in $\supp(\mu)$, and since $\mu$ is generating, the $\supp(\nu)$ is $\Gamma$-invariant. Therefore, by minimality of the action $\Gamma\act X$, we conclude that $\supp(\nu)= X$. This implies every $\mu$-stationary state on $C(X)$ is faithful, hence the result follows from Theorem \ref{main-general}.\qed

\end{document}